\documentclass[12pt]{amsart}
\usepackage{amscd,amsmath,amsthm,amssymb}
\usepackage{mathtools}
\usepackage[margin=2cm]{geometry}
\usepackage{epsfig}
\usepackage{rawfonts}
\usepackage{enumerate}
\usepackage{graphics}
\usepackage{multirow}
\usepackage{xspace}
\usepackage{graphicx}
\usepackage{mathrsfs}
\usepackage{amsmath}
\usepackage{amsfonts}
\usepackage{amssymb}
\usepackage{amsthm}
\usepackage{graphicx}
\usepackage{booktabs}
\usepackage{caption}
\usepackage{listings}
\usepackage{setspace}
\usepackage[mathscr]{eucal}
\usepackage{pgfplots}
\usepackage{hyperref}
\usepackage{wrapfig}
\usepackage{floatflt,epsfig}
\usepackage{ dsfont }
\usepackage{amscd}
\usepackage{tikz-cd}
\usepackage{fancyhdr}
\usepackage[all]{xy}
\usepackage{latexsym}
\usepackage{amscd}
\usepackage{pifont}
\usepackage{subfig}
\usepackage{easyReview}
\usepackage{subfig}
\usepackage{pstricks-add}
\usepackage{pgf,tikz,pgfplots}
\pgfplotsset{compat=1.15}
\usepackage{mathrsfs}
\usetikzlibrary{arrows}
\usetikzlibrary[patterns]
 \usepackage[normalem]{ulem}
 \usepackage{multicol}

\def\NZQ{\Bbb}               

\def\KK{{\NZQ K}}

\renewcommand{\qedsymbol}{$\square$}

\def\opn#1#2{\def#1{\operatorname{#2}}} 
\opn\chara{char} \opn\length{\ell} \opn\pd{pd} \opn\rk{rk}
\opn\projdim{proj\,dim} \opn\injdim{inj\,dim} \opn\rank{rank}
\opn\depth{depth} \opn\grade{grade} \opn\height{height}
\opn\embdim{emb\,dim} \opn\codim{codim}

\opn\Tr{Tr} \opn\bigrank{big\,rank}
\opn\superheight{superheight}\opn\lcm{lcm}
\opn\trdeg{tr\,deg}
	\opn\reg{reg} \opn\lreg{lreg} \opn\ini{in} \opn\lpd{lpd}
	\opn\size{size} \opn\sdepth{sdepth}
	\opn\link{link}\opn\fdepth{fdepth}\opn\lex{lex}\opn\dist{dist}
	\opn\div{div} \opn\Div{Div} \opn\cl{cl} \opn\Cl{Cl}
	\opn\Spec{Spec} \opn\Supp{Supp} \opn\supp{supp} \opn\Sing{Sing}
	\opn\Ass{Ass} \opn\Min{Min}\opn\Mon{Mon}
	\opn\Ann{Ann} \opn\Rad{Rad} \opn\Soc{Soc}
	\opn\Im{Im} \opn\Ker{Ker} \opn\Coker{Coker} \opn\Am{Am}
	\opn\Hom{Hom} \opn\Tor{Tor} \opn\Ext{Ext} \opn\End{End}
	\opn\Aut{Aut} \opn\id{id}
	
	\opn\nat{nat}
	\opn\pff{pf}
	\opn\Pf{Pf} \opn\GL{GL} \opn\SL{SL} \opn\mod{mod} \opn\ord{ord}
	\opn\Gin{Gin} \opn\Hilb{Hilb}\opn\sort{sort}
	\opn\aff{aff} \opn
	\con{conv} \opn\relint{relint} \opn\st{st}
	\opn\lk{lk} \opn\cn{cn} \opn\core{core} \opn\vol{vol}
	\opn\link{link} \opn\star{star}\opn\lex{lex}\opn\set{set}
	\opn\gr{gr}
	
	\def\pot#1#2{#1[\kern-0.28ex[#2]\kern-0.28ex]}

	\opn\dirlim{\underrightarrow{\lim}}
	\opn\inivlim{\underleftarrow{\lim}}
	\def\Implies{\ifmmode\Longrightarrow \else
		\unskip${}\Longrightarrow{}$\ignorespaces\fi}
	\def\implies{\ifmmode\Rightarrow \else
		\unskip${}\Rightarrow{}$\ignorespaces\fi}
	\def\iff{\ifmmode\Longleftrightarrow \else
		\unskip${}\Longleftrightarrow{}$\ignorespaces\fi}

	\let\:=\colon
	\let\epsilon\varepsilon
	\let\kappa=\varkappa
	\def\qed{\ifhmode\textqed\fi
		\ifmmode\ifinner\quad\qedsymbol\else\dispqed\fi\fi}
	\def\textqed{\unskip\nobreak\penalty50
		\hskip2em\hbox{}\nobreak\hfil\qedsymbol
		\parfillskip=0pt \finalhyphendemerits=0}
	\def\dispqed{\rlap{\qquad\qedsymbol}}
	
	\opn\dis{dis}
	\def\pnt{{\raise0.5mm\hbox{\large\bf.}}}
	
	\opn\Lex{Lex}
	
        \newtheorem{Theorem}{Theorem}[section]
	
	\newtheorem{Corollary}[Theorem]{Corollary}
	\newtheorem{Proposition}[Theorem]{Proposition}

\begin{document}

        \title[]{Complementary edge ideals}

\author[T.~Hibi]{Takayuki Hibi}
\address[Takayuki Hibi]
{Department of Pure and Applied Mathematics, 
Graduate School of Information Science and Technology, 
Osaka University, 
Suita, Osaka 565-0871, Japan}
\email{hibi@math.sci.osaka-u.ac.jp}

\author[A. A. Qureshi]{Ayesha Asloob Qureshi}      
   \address[Ayesha Asloob Qureshi]{Sabanci University, Faculty of Engineering and Natural Sciences, Orta Mahalle, Tuzla 34956, Istanbul, Turkey}	
\email{aqureshi@sabanciuniv.edu, ayesha.asloob@sabanciuniv.edu}

\author[S.~Saeedi~Madani]{Sara Saeedi Madani}
\address[Sara Saeedi Madani]
{Department of Mathematics and Computer Science, Amirkabir University of Technology, Tehran, Iran, and School of Mathematics, Institute for Research in Fundamental Sciences, Tehran, Iran} 
\email{sarasaeedi@aut.ac.ir, sarasaeedim@gmail.com}

	
    
	\thanks{} 

        \maketitle

\begin{abstract}
In this paper, we introduce the concept of complementary edge ideals of graphs and study their algebraic properties and invariants. 
\end{abstract}

\section*{Introduction}
Since the birth of Combinatorial Commutative Algebra in 1975, the study of squarefree monomial ideals has been one of the central topics in this area. Among them, the squarefree monomial ideals of degree 2, namely the edge ideals of graphs, have been widely studied. In this paper, we introduce the concept of complementary edge ideals of graphs.

Let $S=\KK[x_1, \ldots, x_n]$ be the polynomial ring over a field $\KK$ in $n\geq 4$ variables. Note that the squarefree monomial ideal $I$ generated in degree $n$ is just a principal ideal, and if $I$ is generated by arbitrary squarefree monomials of degree $n-1$, then $I$ is matroidal, which is well--studied in the literature.  In this paper, we consider the ideals of $S$ generated by an arbitrary set of degree $n-2$ squarefree monomials. These ideals naturally correspond to simple graphs on $n$ vertices. Let $G$ be a simple graph with vertex set $V(G)=[n]=\{1, \ldots, n\}$ and edge set $E(G)$. Then we define the {\em complementary edge ideal} associated to $G$ as
\[
I_c(G)=(x_1\cdots x_n/x_ix_j : \{i,j\} \in E(G)).
\]
In \cite{ALS}, the authors introduced the notion of the generalized Newton complementary dual of a monomial ideal, which provides a general framework for all monomial ideals. In the special case of squarefree monomial ideals of degree two, this construction yields the complementary edge ideals defined above.

If $G$ is a complete graph, then $I_c(G)$ is simply a squarefree Veronese ideal $I_{n,n-2}$ of $S$. It is known \cite[Theorem 12.6.2, Theorem 12.6.7]{HHBook} that $I_{n,n-2}$ has linear resolution and $S/I_{n,n-2}$ is Cohen-Macaulay and hence level, but never Gorenstein. Then, it is natural to ask when $I_c(G)$ admits the aforementioned properties for an arbitrary simple graph $G$. 

A breakdown of the contents of the paper is given as follows. In Section~\ref{sec1}, we consider the Alexander dual of $I_c(G)$ and characterize the sequentially Cohen--Macaulay, Cohen-Macaulay and Gorenstein complementary edge ideals in terms of their underlying graphs. In Section\ref{sec2}, we characterize all graphs $G$ for which $I_c(G)$ admits linear resolution. In Section~\ref{sec3}, we discuss the Betti numbers of $I_c(G)$ and characterize when $I_c(G)$ has pure resolution and when $S/I_c(G)$ is level. Finally in Section~\ref{sec4}, we discuss all possible pairs $(\pd(I_c(G)), \reg(I_c(G)))$.

\section{Alexander dual}\label{sec1}
In this section, we consider the Alexander dual of $I_c(G)$ and characterize when $I_c(G)$ is unmixed, and when $S/I_c(G)$ is Cohen-Macaulay or Gorenstein. 

Let $G$ be a graph on $[n]$ and let $T\subseteq [n]$. We define the monomial prime ideal $P_G(T)$ in $S$ to be the ideal generated by the variables corresponding to the elements of $T$, namely,
\[
P_G(T)=(x_i: i\in T).
\]

We describe the minimal prime ideals of complementary edge ideals in the next theorem. Given a graph $G$, we denote the complementary graph of $G$ by $\overline{G}$. Given any $T\subseteq V(G)$, the induced subgraph of $G$ on $T$ is denoted by $G|_T$. Moreover, the complete graph on $n$ vertices is denoted by $K_n$. 

\begin{Theorem}\label{Min}
    Let $G$ be a graph on $[n]$. Then the minimal prime ideals of $I_c(G)$ are as follows:
    \[
    \Min{I_c(G)}=\{P_G(T): T\subseteq [n], G|_{T}~\textit{is isomorphic to}~K_3~\textit{or}~\overline{K_2}\}.
    \]
    In particular, $\height(I_c(G))=2$ and $S/I_c(G)$ is Cohen--Macaulay if and only if $\pd(I_c(G))=1$, for any noncomplete graph $G$.
\end{Theorem}

\begin{proof}
   Let $T\subseteq [n]$. First suppose that $G|_{T}$ is isomorphic to $\overline{K_2}$. Then $I_c(G)\subseteq P_G(T)$ and since $G$ has no isolated vertex, $P_T$ is a minimal prime ideal of $I_c(G)$. Next suppose that $G|_{T}$ is isomorphic to $K_3$. Then $I_c(G)\subseteq P_G(T)$ and $P_G(T)$ is a minimal prime ideal of $I_c(G)$.

   Now, let $Q=(x_{i_1},\ldots, x_{i_q})$ be a prime ideal containing $I_c(G)$. Assume that $\{i_1,\ldots, i_q\}$ does not contain any two nonadjacnet vertices. Then, the induced subgraph of $G$ on $\{i_1,\ldots, i_q\}$ is a complete graph. In this case it is clear that $q\geq 3$. Thus, $\{i_1,\ldots, i_q\}$ contains the vertex set of a triangle in $G$, and hence the desired result follows. 
\end{proof}

 Recall that an ideal in $S$ is called {\em unmixed} if its minimal prime ideals are of the same height. A graph $G$ is called {\em triangle-free} if it does not have any induced subgraph isomorphic to $K_3$.  
 
 As an immediate consequence of Theorem~\ref{Min}, we have the following characterization for unmixed complementary edge ideals.
   
\begin{Corollary}\label{unmixed}
    Let $G$ be a graph. Then the following are equivalent:
\begin{enumerate}
    \item $I_c(G)$ is unmixed.
    \item $G$ is complete or triangle-free. 
\end{enumerate}
\end{Corollary}

We denote the clique complex of a graph $G$ by $\Delta(G)$. Recall that each facet of $\Delta(G)$ is the vertex set of a maximal clique of $G$. Also, recall that the pure $k$-th skeleton of a simplicial complex $\Delta$, denoted by $\Delta^{[k]}$, is a simplicial complex whose facets are exactly the $k$-dimensional faces of $\Delta$. We denote the facet ideal of a simplex complex $\Delta$ by $I(\Delta)$, and the edge ideal of a graph $G$ by $I(G)$. 

The next corollary provides a combinatorial description for the Alexander dual of $I_c(G)$ in terms of the nonedges and triangles of $G$.    
    
\begin{Corollary}\label{dual}
    Let $G$ be a graph. Then the Alexander dual $I_c(G)^{\vee}$ of $I_c(G)$ is as follows:
\[
I_c(G)^{\vee}=I(\overline{G})+I({\Delta(G)}^{[2]}). 
\]
\end{Corollary}

Now, we provide the classification of all graphs for which  $I_c(G)^{\vee}$ is componentwise linear and equivalently $S/I_c(G)$ is sequentially Cohen--Macaulay. Recall that a graph is called a {\em chordal} graph if it has no induced cycle of length bigger than~3. For a squarefree monomial ideal $I$ in $S$, we denote by $I_{[j]}$ the monomial ideal generated by all squarefree monomials of degree~$j$ in $I$. 

\begin{Corollary}\label{componentwise}
    Let $G$ be a graph. Then the following are equivalent:
\begin{enumerate}
    \item $I_c(G)^{\vee}$ is componentwise linear.
    \item $S/I_c(G)$ is sequentially Cohen--Macaulay.
    \item $G$ is chordal. 
\end{enumerate}
\end{Corollary}

\begin{proof}
It is well--known that (1) and (2) are equivalent \cite[Theorem~8.2.20]{HHBook}. Now, we show that (1) and (3) are equivalent. For simplicity, let $I=I_c(G)^{\vee}$. Then, by Corollary~\ref{dual}, we have $I_{[2]}=I(\overline{G})$ which has linear resolution if and only if $G$ is chordal, by Fr\"oberg theorem \cite{F}. On the other hand, by Corollary~\ref{dual}, $I_{[3]}$ is the squarefree monomial ideal generated by all degree~3 squarefree monomials in $S$, and hence it is a squarefree Veronese ideal which has linear resolution. Thus, it follows from \cite[Proposition~8.2.17]{HHBook} that $I$ is componentwise linear if and only if $G$ is chordal, as desired.     
\end{proof}

Next, we provide the classification of all graphs for which $S/I_c(G)$ is Cohen--Macaulay. Recall that a graph is called a {\em forest} if it has no cycles, and a connected forest is called a {\em tree}. 

\begin{Corollary}\label{CM}
    Let $G$ be a graph. Then the following are equivalent:
\begin{enumerate}
    \item $S/I_c(G)$ is Cohen--Macaulay.
    \item $G$ is complete or a forest. 
\end{enumerate}
\end{Corollary}

\begin{proof}
By \cite[Lemma~3.6]{FV}, $S/I_c(G)$ is Cohen--Macaulay if and only if $S/I_c(G)$ is sequentially Cohen--Macaulay and $I_c(G)$ is unmixed. This is the case if and only if $G$ is complete or a forest, by Corollary~\ref{unmixed} and Corollary~\ref{componentwise}, since the only triangle--free chordal graphs are forests.    
\end{proof}

Next, we determine when $S/I_c(G)$ is Gorenstein. 

\begin{Corollary}\label{Gorenstein}
    For a given graph $G$, the following are equivalent:
    \begin{enumerate}
    \item $S/I_c(G)$ is Gorenstein.
    \item $G$ is the disjoint union of two edges. 
\end{enumerate}
\end{Corollary}

\begin{proof}
    (1) \implies (2) Suppose that $S/I_c(G)$ is Gorenstein, and hence $G$ is not complete. Then $S/I_c(G)$ is Cohen--Macaulay, and hence $I_c(G)^{\vee}$ has linear resolution by \cite{ER}. Thus, $\reg(I_c(G)^{\vee})=2$, by Corollary~\ref{dual} and Corollary~\ref{CM}. Therefore, by \cite[Proposition 8.1.10]{HHBook}, we have $\pd(I_c(G))=1$. Then, Hilbert-Burch theorem \cite[Lemma 9.2.4]{HHBook} implies that $I_c(G)$ is minimally generated by exactly two monomials, since $S/I_c(G)$ is Gorenstein and hence the last Betti number is equal to~$1$. Thus, $G$ has exactly two edges. Then $G$ is the disjoint union of two edges, since $n\geq 4$. 

    (2) \implies (1) If $G$ is the disjoint union of two edges, then $S/I_c(G)$ is complete intersection, and hence Gorenstein. 
\end{proof}


\section{Linear resolution}\label{sec2}

In this section, we characterize the complementary edge ideals with linear resolution. For this purpose, we recall some notion from \cite{BHZ}. Let $I$ be a monomial ideal in $S$ generated in degree~$d$. The graph $G_I$ associated to $I$ is defined as a graph whose vertex set is just the minimal set of monomial generators of $G$, denoted by $\mathcal{G}(I)$ and its edge set is
	\[
	E(G_I)=\big{\{}\{u,v\}: u,v\in \mathcal{G}(I)~\text{with}~\deg (\lcm(u,v))=d+1 \big{\}}.
	\]
	For any $u, v\in \mathcal{G}(I)$, we denote by  $G^{(u,v)}_I$ the induced subgraph of $G_I$ on the vertex set
	\[
	V(G^{(u,v)}_I)=\{w\in \mathcal{G}(I)\: \text{$w$ divides $\lcm(u, v)$}\}.
	\]
	
	\medskip
	Let $I\subset S$ be a graded ideal generated in a single degree. Then, $I$ is said to be {\em linearly related}, if the first syzygy module of $I$ is generated by linear relations. 
    
\medskip 
We benefit from the following result to prove the next theorem.
	
	\begin{Theorem} \label{lcm graph}
		\cite[Corollary~2.2]{BHZ}
        Let $I$ be a monomial ideal generated in degree $d$. Then the following are equivalent:
        \begin{enumerate}
            \item $I$ is linearly related.
            \item for any $u,v\in \mathcal{G}(I)$, there exists a path in $G^{(u,v)}_I$ connecting $u$ and $v$.	
        \end{enumerate}  
	\end{Theorem}

Let $I$ be a monomial ideal in $S$. Then $I$ is said to have {\em linear quotients} if the minimal monomial generators of $I$ can be ordered $u_1, \ldots, u_r$ such that $(u_1, \ldots, u_{i-1}):u_i$ is generated by variables for each $i=2, \ldots, r$. Such an ordering for the elements of $\mathcal{G}(I)$ is called a {\em linear quotient order}. Given a monomial $u$ in $S$, we denote the set of variables which divide $u$ by $\supp(u)$. Also, given monomials $u,v\in S$, we set $u:v=\lcm(u,v)/v$. Let $i$ be a vertex in $G$. Then neighborhood of $i$ in $G$, denoted by $N_G(i)$ is the set of all vertices adjacent to $i$. 

\medskip 
Now we present the main theorem of this section.

\begin{Theorem}\label{linear}
    Let $G$ be a graph. Then the following are equivalent:
\begin{enumerate}
\item $I_c(G)$ has linear quotients. 
    \item $I_c(G)$ has linear resolution.
    \item $I_c(G)$ is linearly related.
    \item $G$ is connected. 
\end{enumerate}
\end{Theorem}

\begin{proof}
  (1)\implies (2) \implies (3) are well--known. 
  

(3) $\implies$ (4) 
Recall that in $I = I_c(G)$, each minimal generator corresponds to an edge $e = \{i,j\}$ of $G$ via the monomial $u_e = x_1 \cdots x_n/x_i x_j$.  For $u_e, u_f \in \mathcal{G}(I)$, the definition of $G_I$ shows that $u_e$ and $u_f$ are adjacent in $G_I$ if and only if  $\supp(u_e)$ and $\supp(u_f)$ differ in exactly one variable. Equivalently, the corresponding edges $e$ and $f$ in $G$ share a vertex. 
Since $I$ is linearly related, Theorem~\ref{lcm graph} implies that for any $u_e, u_f \in \mathcal{G}(I)$ there exists a path in $G^{(u_e,u_f)}_I$ connecting $u_e$ to $u_f$.  Hence, given any two edges $e$ and $f$ of $G$, there are edges $e_1, e_2, \ldots, e_r$ of $G$ such that $u_e, u_{e_1}, \ldots, u_{e_r},u_f$ form a path in $G^{(u_e,u_f)}_I$. By the adjacency correspondence above, $e, e_1, \ldots, e_r, f$ is a path in $G$ connecting $e$ and $f$. This proves that $G$ is connected.



(4) \implies (1)  We now describe an explicit ordering of the generators of $I_c(G)$ which yields linear quotients. Pick a vertex $a \in V(G)$, and let $E(a) = \{ e_1, \ldots, e_{r} \}$ be the set of edges of $G$ incident with $a$. We place the corresponding monomials $u_{e_1}, \ldots, u_{e_r}$ first in our order. Since, $\supp(u_{e_i})$ and $\supp(u_{e_{j}})$ differ by each other at one variable for all $1\leq i < j \leq r$, it follows that $u_{e_1}, \ldots, u_{e_r}$ is a linear quotient order. If $E(G)=E(a)$, the proof is complete. Otherwise, choose a vertex $b \in N_G(a)$, where $N_G(a)$ denotes the neighbor set of $a$. Let $E(b) = \{ f_1, \ldots, f_{s} \}$ be the set of edges incident with $b$, excluding $\{a,b\}$. We place the corresponding monomials $u_{f_1}, \ldots, u_{f_{s}}$ next in the order. 
Suppose we have already listed all monomials corresponding to edges incident with vertices in a set $T \subseteq V(G)$.  
If there remain edges not yet listed, then due to connectivity of $G$, there exists a vertex $c$ which is neighbor of some vertex in $T$ such that at least one incident edge of $c$ has not yet been listed. Choose such a vertex $c$. List the monomials corresponding to all such edges (excluding those whose other endpoint already lies in $T$).  Add $c$ to $T$ and repeat until all edges have been exhausted.

\medskip
This process produces an ordering of elements of $\mathcal{G}(I_c(G))$. Next we show that this ordering is a linear quotient ordering. Let $u_e$ and $u_f$ be two monomials in $\mathcal{G}(I_c(G))$ such that $u_e$ appears before $u_f$. 

First assume that $e$ and $f$ are adjacent and let $\ell$ be the vertex incident to $f$ and non-incident to $e$. Then $u_{e}: u_{f}=x_\ell$ as required. 

Next assume that $e$ and $f$ are disjoint. Let $e=\{i,j\}$ and $f=\{k, \ell \}$. Then $u_{e}: u_{f}=x_k x_\ell$. Due to our choice of ordering, there exists some edge $g$ such that in our ordering $u_g$ is listed before $u_f$ and $g$ is adjacent to $f$. Let $g=\{k, m\}$. Then  $(u_{g}): u_{f}=x_\ell$, as required. 
\end{proof}


\section{Betti numbers of forests}\label{sec3}

In this section, we investigate the Betti numbers of the complementary edge ideals and characterize all graphs $G$ for which $S/I_c(G)$ is level or $I_c(G)$ has pure resolution. 

\begin{Proposition}\label{nonvanishing Betti numbers}
 Let $G$ be a graph. If $\beta_{i,j}(I_c(G))\neq 0$, then 
 \[
 (i,j)\in \{(0,n-2), (1,n-1), (1,n), (2,n)\}.
 \]
\end{Proposition}

\begin{proof}
 Note that by Hochster's formula, the maximum possible shift in the minimal graded free resolution of $I_c(G)$ is $n$. Since $I_c(G)$ is generated in degree~$n-2$, the possible non--vanishing graded Betti numbers of $I_c(G)$ are $\beta_{0,n-2}$, $\beta_{1,n-1}$, $\beta_{1,n}$ and $\beta_{2,n}$.    
\end{proof}

The non--vanishing graded Betti numbers of the complementary edge ideal of a tree are known by Theorem~\ref{linear} and Corollary~\ref{CM}. As a next natural case, we consider disconnected forests in the following. For this purpose, we need the description of the Stanley--Reisner simplicial complex of $I_c(G)$. We define the simplicial complex $\Gamma_G$ to be the simplicial complex whose set of facets is as follows:
\[
\big{\{}[n]-\{i,j\}: i,j\in [n], \{i,j\}\notin E(G)\big{\}}\cup \big{\{}[n]-\{i,j,k\}:  \{i,j,k\}\in {\Delta(G)}^{[2]}\big{\}} 
\]
Then, it can be seen that the Stanley--Reisner ideal of $\Gamma_G$ is $I_c(G$).

\medskip
Beside the Stanley--Reisner simplicial complex of $I_c(G)$, the simplicial complex whose facets correspond to the minimal generators of $I_c(G)$ is also useful for us. We denote this simplicial complex by $\Lambda_G$. Then $I_c(G)$ is the facet ideal of $\Lambda_G$, namely $I_c(G)=I(\Lambda_G)$. 

We are going to use a sufficient condition for non--vanishing Betti numbers of  facet ideals given in \cite{EF}. For this, we need to recall the notion of a special type of covers for facets.   

Let $\Delta$ be a simplicial complex. Then, a subset $\mathcal{C}$ of facets of $\Delta$ is called a {\em facet cover} of $\Delta$ if every vertex $i$ of $\Delta$ belongs to some facet $F$ in $\mathcal{C}$. Moreover, a facet cover is called \emph{minimal} if no proper subset of it is a facet cover of $\Delta$. A sequence $F_1,\ldots,F_k$ of facets of $\Delta$ is called a {\em well-ordered facet cover} if $\{F_1,\ldots,F_k\}$ is a minimal facet cover of $\Delta$ and for every facet $H\notin \{F_1,\ldots,F_k\}$ of $\Delta$ there exists $i\leq k-1$ with $F_i \subseteq H \cup F_{i+1} \cup F_{i+2}\cup \cdots \cup F_k.$

A {\em subcollection} of $\Delta$ is a simplicial complex whose facets are also facets of $\Delta$. Let $W\subseteq V(\Delta)$. Then an \emph{induced subcollection} $\Delta^W$ of $\Delta$ on $W$ is the simplicial complex whose facets are those facets of $\Delta$ which are contained in $W$. If $u$ is a squarefree monomial in $S$, then we denote by $\Delta^u$ the induced subcollection $\Delta^{\supp(u)}$ of $\Delta$. As usual, we denote the {\em induced subcomplex} of $\Delta$ on $W$ by $\Delta_W$. 

\begin{Theorem}\label{cover}
\cite[Corollary 3.4]{EF}
	Let $\Delta$ be a simplicial complex and let $u$ be a squarefree monomial. If $\Delta^{u}$ has a well--ordered facet cover of cardinality $i$, then $\beta_{i-1,u}(I(\Delta)) \neq 0$.
\end{Theorem}

In the next theorem we determine all non--vanishing graded Betti numbers of complementary edge ideals of disconnected forests. 

\begin{Theorem}\label{Betti numbers of forests}
   Let $G$ be a disconnected forest. Then the only non--vanishing graded Betti numbers of $I_c(G)$ are as follows:
   \begin{enumerate}
        \item $\beta_{0,n-2}(I_c(G))=n-c$, where $c$ is the number of connected components of $G$. 
       \item $\beta_{1,n}(I_c(G))\neq 0$. 
       \item $\beta_{1,n-1}(I_c(G))\neq 0$ if and only if $G$ is not a disjoint union of edges. 
   \end{enumerate} 
\end{Theorem}

\begin{proof}
Note that since $S/I_c(G)$ is Cohen-Macaulay, and hence $\pd(I_c(G))=1$, we have $\beta_{2,n}(I_c(G))=0$, and we only need to consider $\beta_{0,n-2}(I_c(G))$, $\beta_{1,n-1}(I_c(G))$ and $\beta_{1,n}(I_c(G))$ in the following:

   (1) Since a forest with $c$ connected components has $n-c$ edges, $I_c(G)$ is minimally generated by $n-c$ monomials, and hence $\beta_{0,n-2}(I_c(G))=n-c$.

   (2) Since $I_c(G)$ is not linearly related by Theorem~\ref{linear}, it follows that $\beta_{1,n}(I_c(G))\neq 0$, by Proposition~\ref{nonvanishing Betti numbers}.    

   (3) First suppose that $G$ is a disjoint union of edges. Then, we show that $\beta_{1,n-1}(I_c(G))\neq 0$. By Hochster's formula, we have
   \[
   \beta_{1,n-1}(I_c(G))=\sum_{W\subseteq [n],|W|=n-1} \dim_{\KK} {\Tilde{H}}_{n-4} ({(\Gamma_G)}_W;\KK). 
   \]
   Since $G$ has no triangles, the only facets of ${\Gamma}_G$ are the complements of edges of $G$. Let $G$ consist of the disjoint edges $e_1,e_2, \ldots,e_t$ with $t\geq 2$, and let $e_1=\{1,2\}$. By the symmetry of $G$, without loss of generality, we assume that $W=[n]-\{1\}$. We claim that $(\Gamma_G)_{W}$ is a cone with the apex $2$, i.e. all facets of $(\Gamma_G)_{W}$ contain the vertex~$2$.  
   For this, we distinguish three possible types of nonedges in $G$:
   \begin{enumerate}
       \item [(i)] $\{i,j\}\subseteq [n]$ with $\{i,j\}\notin E(G)$ and $i,j\notin \{1,2\}$: the complements of such nonedges contain the vertex~$2$, and hence after removing $1$ from them, they are faces of ${(\Gamma_G)}_W$ containing ~$2$. 
       \item [(ii)] $\{1,j\}$ with $j\in [n]$ and $\{1,j\}\notin E(G)$: the complements of such nonedges contain the vertex~$2$, and hence they are faces of ${(\Gamma_G)}_W$ containing ~$2$, since they do not contain 1.  
       \item [(iii)] $\{2,j\}$ with $j\in [n]$ and $\{2,j\}\notin E(G)$: after removing $1$ from the complements of such nonedges, they are contained in some faces of type~(ii) of ${(\Gamma_G)}_W$ which also contain ~$2$. 
   \end{enumerate}
   Therefore, the vertex~$2$ is contained in all the facets of ${(\Gamma_G)}_W$, and hence ${(\Gamma_G)}_W$ is a cone, as we claimed. Therefore, ${\Tilde{H}}_{n-4} ({(\Gamma_G)}_W;\KK)$ vanishes. This together with Hochster's formula implies that $\beta_{1,n-1}(I_c(G))=0$, as desired. 

   Conversely, suppose that $G$ is not a disjoint union of edges. Then $G$ has a connected component containing two distinct edges $\{i,j\}$ and $\{j,k\}$. Let $u=x_1\cdots \hat{x_j}\cdots x_n$ be the squarefree monomial of degree~$n-1$. We put $F_1=[n]-\{j,k\}$ and $F_2=[n]-\{i,j\}$. Then, $\{F_1,F_2\}$ is clearly a minimal facet cover of cardinality~$2$ for the simplicial complex $(\Lambda_G)^u$. Now, let $H\neq F_1, F_2$ be a facet of $(\Lambda_G)^u$. Then, $H=[n]-\{s,j\}$ where $\{s,j\}\in E(G)$ and $s\neq i, k$. Then, we have $F_1\subseteq H\cup F_2$, since $i\in H$. Thus, $\{F_1,F_2\}$ is also a well--covered facet cover of cardinality~$2$ for $(\Lambda_G)^u$. Therefore, we deduce by Theorem~\ref{cover} that $\beta_{1,n-1}(I_c(G))\neq 0$ which completes the proof.       
\end{proof}

 Let $I$ be a graded ideal in $S$. Recall that the minimal graded free resolution of $I$ is \emph{pure} if in all of its steps only one single shift degree appears. Now we give a characterization for complementary edge ideals with pure resolution.

\begin{Corollary}\label{Pure resolution}
    Let $G$ be a graph. Then the following are equivalent:
    \begin{enumerate}
        \item $I_c(G)$ has pure resolution.
        \item $G$ is connected or a disjoint union of edges. 
    \end{enumerate}
\end{Corollary}

\begin{proof}
   (1)\implies (2) Suppose that $I_c(G)$ has pure resolution. If $G$ is not connected, then we show that it a disjoint union of edges. By Theorem~\ref{linear}, $I_c(G)$ does not have linear resolution, and hence $\beta_{1,n}(I_c(G))\neq 0$. Therefore, $\beta_{1,n-1}(I_c(G))=0$, since $I_c(G)$ has pure resolution. Thus, $G$ is a disjoint union of edges by Theorem~\ref{Betti numbers of forests}, and hence (2) follows. 

   (2)\implies (1) If $G$ is connected, then by Theorem~\ref{linear} $I_c(G)$ has linear resolution and hence it has pure resolution. If $G$ is a disjoint union of edges, then Propositon~\ref{nonvanishing Betti numbers} and Theorem~\ref{Betti numbers of forests}, imply that $I_c(G)$ has pure resolution.
\end{proof}

 Recall that $S/I$ is {\em level} if it is Cohen--Macaulay and the last step of the minimal graded free resolution of $S/I$ is pure. Next, we give a characterization for $S/I_c(G)$ with level property. Let $I$ be a graded ideal in $S$.

\begin{Corollary}\label{Level}
    Let $G$ be a graph. Then the following are equivalent:
    
    \begin{enumerate}
        \item $S/I_c(G)$ is level.
        \item $G$ is a complete graph, a tree or a disjoint union of edges. 
    \end{enumerate}

\end{Corollary}

\begin{proof}
    (1)\implies (2) Suppose that $S/I_c(G)$ is level. Then, $S/I_c(G)$ is Cohen--Macaulay, and hence $\pd(I_c(G))=1$ and $G$ is a complete graph or a forest by Corollary~\ref{CM}. If $G$ is disconnected but not a disjoint union of edges, then $\beta_{1,n-1}(I_c(G))\neq 0$ by Theorem~\ref{Betti numbers of forests}. On the other hand, by the same theorem, we have $\beta_{1,n}(I_c(G))\neq 0$ which contradicts purity of the minimal graded free resolution of $I_c(G)$, and hence levelness of $S/I_c(G)$. 

    (2)\implies (1) If $G$ is a complete graph, then $I_c(G)$ is a squarefree Veronese ideal, and hence $S/I_c(G)$ is level. If $G$ is a tree, then $I_c(G)$ has linear resolution by Theorem~\ref{linear}, and hence it has pure resolution, and also in this case $S/I_c(G)$ is Cohen--Macaulay by Corollary~\ref{CM}. Thus, $S/I_c(G)$ is level. If $G$ is a disjoint union of edges, then $S/I_c(G)$ is Cohen--Macaulay by Corollary~\ref{CM}, and $I_c(G)$ has pure resolution by Corollary~\ref{Pure resolution}, and hence $S/I_c(G)$ is level.     
\end{proof}


\section{Projective dimension and regularity}\label{sec4}

In this section, we consider all possible pairs $(\pd(I_c(G)), \reg(I_c(G)))$. To do this, we begin with the following proposition.  

\begin{Proposition}\label{pd and reg}
    Let $G$ be a graph. Then 
\[
1\leq \pd(I_c(G))\leq 2,
\]
\[
n-2\leq \reg(I_c(G))\leq n-1.
\]
\end{Proposition}

\begin{proof}
   It follows from Proposition~\ref{nonvanishing Betti numbers} that $\pd(I_c(G))\leq 2$. On the other hand, since $G$ has at least four vertices and no isolated vertices, it has at least two edges. Thus, $I_c(G)$ is minimally generated by at least two monomials, and hence $\pd(I_c(G))\neq 0$. Therefore, we have 
  \[
  1\leq \pd(I_c(G))\leq 2.
  \]
  According to the possible non--vanishing Betti numbers mentioned in Proposition\ref{nonvanishing Betti numbers}, it also follows that 
  \[
  n-2\leq \reg(I_c(G))\leq n-1.
  \]
\end{proof}

We conclude this paper, with the following classification theorem. 

\begin{Theorem}
    Let $G$ be a graph on $n$ vertices. Then 
    \begin{enumerate}
      \item $(\pd (I_c(G)), \reg(I_c(G)))= (1, n-2)$ if and only if $G$ is a tree or a complete graph. 
     \item $(\pd (I_c(G)), \reg(I_c(G)))= (1, n-1)$ if and only if $G$ is a disconnected forest. 
     \item $(\pd (I_c(G)), \reg(I_c(G)))= (2, n-2)$ if and only if $G$ is a connected noncomplete graph with at least one cycle. 
           \item $(\pd (I_c(G)), \reg(I_c(G)))= (2, n-1)$ if and only if $G$ is disconnected with at least one cycle. 
    \end{enumerate}
\end{Theorem}

\begin{proof}
    First suppose that $I_c(G)$ has linear resolution. Then, $G$ is connected and $\reg(I_c(G))=n-2$, by Theorem~\ref{linear}. Thus, $G$ is a tree or a complete graph if and only if $S/I_c(G)$ is Cohen--Macaulay by Corollary~\ref{CM}, and if and only if $\pd(I_c(G))=1$. This together with Proposition~\ref{pd and reg} implies (1) and (3). 

    Next, suppose that $I_c(G)$ does not have linear resolution. Then, $G$ is disconnected and $\reg(I_c(G))=n-1$, by Theorem~\ref{linear} and Proposition~\ref{pd and reg}. Thus, $G$ is a disconnected forest if and only if $S/I_c(G)$ is Cohen--Macaulay by Corollary~\ref{CM}, and if and only if $\pd(I_c(G))=1$. This together Proposition~\ref{pd and reg} implies (2) and (4). 
\end{proof}

It would be of interest to study the ideals generated by arbitrary squarefree monomials of degree $n-3$ in $n$ variables.

\section*{Acknowledgment}

The present paper was completed while the first and third authors stayed at Sabanc\i~\"Universitesi, Istanbul, Turkey, August~05 to~14, 2025. Ayesha Asloob Qurehsi was supported by Scientific and Technological Research Council of Turkey T\"UB\.{I}TAK under the Grant No: 124F113. Sara Saeedi Madani was in part supported by a grant from IPM (No. 1404130019). We thank Adam Van Tuyl for bringing to our attention, after the first version of this paper appeared on arXiv, the related work of Ansaldi, Lin, and Shen on Newton complementary duals of monomial ideals.


\begin{thebibliography}{999}
\bibitem[ALS]{ALS} K. Ansaldi, K. Lin, Y. Shen, \textit{Generalized Newton complementary duals of monomial ideals}, J. Algebra Appl., {\bf 20} (2021), no. 2, 2150021.
\medskip

\bibitem[BHZ]{BHZ}
	M. Bigdeli, J. Herzog, R. Zaare-Nahandi, \textit{On the index of powers of edge ideals}, Comm. Algebra, {\bf 46}  (2018), 1080--1095.


\medskip
	\bibitem[ER]{ER}
	J.A.~Eagon, V.~Reiner,  \textit{Resolutions of Stanley--Reisner rings and Alexander duality}, J. Pure Appl. Algebra, {\bf 130} (1998), 265-275.

	
\medskip
	\bibitem[EF]{EF}
	N.~Erey, S.~Faridi,  \textit{Betti numbers of monomial ideals via facet covers}, J. Pure Appl. Algebra, {\bf 220} (2016), no. 5, 1990--2000.

\medskip
	\bibitem[FV]{FV}
C.A.~Francisco, A.~Van Tuyl, \textit{Sequentially Cohen--Macaulay edge ideals}, Proc. Amer. Math. Soc. (2007), no. 8, 2327--2337.

	\medskip
	\bibitem[F] {F}
	R.~Fr\"{o}berg, \textit{On Stanley-Reisner rings}, Topics in algebra, Banarch Center Publications,
	{\bf  26} (2) (1990),  57--70.
    
	
	\medskip
	\bibitem[HHBook]{HHBook}
	J.~Herzog, T.~Hibi, \textit{Monomial Ideals}, Graduate Text in Mathematics, Springer, 2011.
	
	
  
	\end{thebibliography}
	\end{document}